\documentclass[]{amsart}
\usepackage{amsaddr}
\usepackage{amsthm}%
\usepackage{amsmath,amssymb,amsfonts,mathabx}%
\usepackage{amsthm}
\usepackage{thmtools}
\usepackage{mathrsfs}%
\usepackage[hidelinks]{hyperref}
\usepackage[capitalize]{cleveref}
\usepackage{graphicx}%
\usepackage{multirow}%
\usepackage[title]{appendix}%
\usepackage{xcolor}%
\usepackage{textcomp}%
\usepackage{manyfoot}%
\usepackage{booktabs}%
\usepackage{algorithm}%
\usepackage{algorithmicx}%
\usepackage{algpseudocode}%
\usepackage{listings}%
\usepackage{orcidlink} 

\usepackage{hyperref}
\usepackage{bm}
\usepackage{dsfont}
\usepackage[capitalize]{cleveref}
\usepackage{pgfplotstable}

\usepackage[backend=biber,bibencoding=utf-8]{biblatex}
\pgfplotsset{compat=1.18} 

\newcounter{mythm}
\crefalias{mythm}{section}
\numberwithin{mythm}{section}

\declaretheorem[style=plain, sibling=mythm]{theorem}

\declaretheorem[style=plain, sibling=mythm]{proposition}
\declaretheorem[style=plain, sibling=mythm]{corollary}

\declaretheorem[style=definition, numberwithin=section]{definition}

\newcommand{\myorcid}[1]{\href{https://orcid.org/#1}{\includegraphics[height=10pt]{orcid.png}}}

\newcommand{\zeps}[3]{Z_{#1}(#2,#3)}

\newcommand{\sums}{\sideset{}{'} \sum}

\def\Xint#1{\mathchoice
   {\XXint\displaystyle\textstyle{#1}}%
   {\XXint\textstyle\scriptstyle{#1}}%
   {\XXint\scriptstyle\scriptscriptstyle{#1}}%
   {\XXint\scriptscriptstyle\scriptscriptstyle{#1}}%
   \!\int}
\def\XXint#1#2#3{{\setbox0=\hbox{$#1{#2#3}{\int}$}
     \vcenter{\hbox{$#2#3$}}\kern-.5\wd0}}
\def\ddashint{\Xint=}

\renewcommand{\Re}[1]{\mathrm{Re}(#1)}
\renewcommand{\Im}[1]{\mathrm{Im}(#1)}

\definecolor{myblue}{rgb}{0, 0.2, 1.0}

\newcommand{\new}[1]{{\textcolor{black}{#1}}}

\begin{document}

\title[Epstein zeta method for many-body lattice sums]{Epstein zeta method for many-body lattice sums}

\author[Buchheit]{Andreas A. Buchheit\orcidlink{0000-0003-4004-713X}}
\address{%
Department of Mathematics, Saarland University, Campus E1.1, Saarbrücken, 66123, Saarland, Germany \\
Department of Mathematics, ETH Zürich, Rämistrasse 101, Zürich, 8092, Switzerland
}

\author[Busse]{Jonathan K. Busse\orcidlink{0009-0001-3323-3455}}
\address{%
Institute of Software Technology, High-Performance Computing Department,German Aerospace Center (DLR), Linder Höhe, Cologne, 51147, North Rhine-Westphalia, Germany \\
Department of Mathematics, Saarland University, Campus E1.1, Saarbrücken, 66123, Saarland, Germany
}


\begin{abstract}
Many-body interactions arise naturally in the perturbative treatment of classical and quantum many-body systems and play a crucial role in the description of condensed matter systems. In the case of three-body interactions, the Axilrod-Teller-Muto (ATM) potential is highly relevant for the quantitative prediction of material properties. This work solves the long-standing issue of the numerical computation of the resulting energies in $d$-dimensional lattice systems. We present an efficiently computable representation of many-body lattice sums in terms of singular integrals over products of Epstein zeta functions. For three-body interactions in three dimensions, this approach reduces the runtime for computing the ATM lattice sum from weeks to minutes. Our approach further extends to a broad class of $n$-body lattice sums. We demonstrate that the computational cost of our method only increases linearly with $n$, evading the exponential increase in complexity of direct summation. We discuss techniques for numerically computing the arising singular integrals and compare the accuracy of our results against computable special cases and against direct summation in low dimensions, achieving full precision for exponents greater than the system dimension.  
Finally, we apply our method to study the stability of a three-dimensional lattice system with Lennard-Jones two-body interactions under the inclusion of an ATM three-body term at finite pressure, finding a transition from the face-centered-cubic to the body-centered-cubic lattice structure with increasing ATM coupling strength. This work establishes both the numerical and analytical foundation for an ongoing investigation into the influence of many-body interactions on the stability of matter.
\end{abstract}

\maketitle

\section{Introduction}

Predicting the properties of exotic materials, given their chemical composition, is a central goal of computational quantum chemistry and condensed matter physics. Using many-body perturbation theory, the energy per particle can be expanded in a formally exact series, where many-body interactions arise as corrections to additive two-body contributions \cite{muser2023interatomic}. These corrections are known to be important for ultracold crystals of noble gases, such as solid Argon, where they can contribute up to 9\,\% of the cohesive energy \cite{schwerdtfeger2016towards}. It has further been suggested that three-body interactions can contribute up to $51\,\%$ of the binding energy of bilayer graphene \cite{anatole2010two}. Even if the spatial dependency of the many-body interaction potential has been determined, for instance by coupled-cluster theory \cite{smits2020first}, the evaluation of the resulting energy is highly challenging as a high-dimensional singular lattice sum, \new{which converges only slowly with increasing summation cutoff,} needs to be computed \cite{hermannConvergenceManybodyExpansion2007}.

In third-order perturbation theory, an effective interaction between trimers of atoms arises. 
The leading order contribution to the three-body potential at large distances due to quantum dipole fluctions has been independently derived by Axilrod and Teller 
\cite{axilrod1943interaction}, and Muto \cite{muto1943force} (in Japanese). A detailed derivation can be found in Refs.~\cite{bade1957drude,axilrod1951triple,bell1970multipolar}.
\begin{definition}[Axilrod-Teller-Muto potential and cohesive energy]
\label{def:atm}
Consider a trimer of particles at positions $\bm r^{(1)},\bm r^{(2)},\bm r^{(3)}\in \mathds R^d$ with relative distance vectors $\bm x= \bm r^{(2)}-\bm r^{(1)}$, $\bm y= \bm r^{(3)}-\bm r^{(1)}$ and $\bm z=\bm r^{(3)}-\bm r^{(2)}$. The ATM potential then reads
\[
U^{(3)}_{\mathrm{ATM}}(\bm x,\bm y) = \bigg(\vert \bm x \vert^{-3}\vert \bm y \vert^{-3}\vert \bm  z\vert^{-3} - 3 \frac{(\bm x\cdot \bm y)(\bm y\cdot \bm z)(\bm z\cdot \bm x)}{\vert \bm x \vert^{5}\vert \bm y \vert^{5}\vert \bm  z\vert^{5}}\bigg)\bigg\vert_{\bm z=\bm y-\bm x}
\]
where $|\bm \cdot|$ denotes the Euclidean norm.
 The resulting cohesive energy per particle \new{for a lattice $\Lambda=A\mathds Z^d$, with $A\in \mathds R^{d\times d}$ regular,} takes the form
\[
E_{\mathrm{coh}}^{(3)} = \frac{1}{6}\,\sideset{}{'}\sum_{\bm x,\bm y \in  \Lambda} U^{(3)}_{\mathrm{ATM}}(\bm x,\bm y),
\]
where the primed sum excludes the ill-defined cases $\bm  x= \bm 0$, $\bm y = \bm  0$, and $\bm x=\bm y$ and where the prefactor $1/6$ avoids double counting. The sum converges absolutely in dimensions $d\in \{1,2,3\}$.
\end{definition} 

\new{In the following, we use the primed sum to exclude all ill-defined summands, namely those for which expressions of the form $\vert \bm{\cdot} \vert^{-\nu}$ would be evaluated at zero.}

The evaluation of the arising high-dimensional and slowly converging lattice sums to a few digits can take up to 4 weeks on a single CPU \cite{schwerdtfeger2016towards}. Many-body interactions beyond the three-body term have been discussed in detail in \cite{bade1957drude} and are relevant for material properties \cite{schwerdtfeger2016towards}, yet the evaluation of the arising lattice sums quickly becomes impossible due to the exponential increase in complexity with number of interacting particles $n$.  We here note recent developments in modelling many-body interactions using cluster expansions for short-range interactions  \cite{bochkarev2024graph}.

This work solves the above issue and provides a numerically efficient method for computing a general class of many-body lattice sums. Importantly, our method allows us to compute general power-law three-body interaction energies such as the ATM cohesive energy, providing full precision for any lattice, in any dimension, and a very general class of potentials. Our method serves as the analytical and numerical foundation for Ref.~\cite{roblesnavarro2025exact}, where the influence of three-body interactions on the stability of cuboidal crystal lattices is investigated.

This article is intended for an interdisciplinary audience with different interests and aims, ranging from applied mathematicians over theoretical chemists to condensed matter and high-energy physicists. For this reason, we have structured it as follows. In Sec.~\ref{sec:zeta_representation}, we define many-body zeta functions as meromorphic continuations of iterated high-dimensional lattice sums. We then show that these zeta functions exhibit an efficiently computable representation in the form of integrals over products of Epstein zeta functions. The Axilrod-Teller\new{-Muto} (ATM) lattice sum, as well as other many-body lattice sums, then follow as a recombination of these zeta functions. In Sec.~\ref{sec:num_int}, we present techniques that allow for an efficient evaluation of the arising singular integrals.  In Sec.~\ref{sec:benchmarks}, we benchmark the precision of our approach against analytically computable formulas and direct summation in low dimensions. Importantly, we show that the computational cost of our method for $n$-body lattice sums only scales linearly with $n$, avoiding the exponential scaling of direct summation and thus allowing for the efficient evaluation of lattice sums of very large dimension. Finally, we provide new results for the contribution of three-body terms to the cohesive energy of crystals in Sec.~\ref{sec:application}. We draw our conclusions and provide an outlook in Sec.~\ref{sec:outlook}.

\section{Zeta representation of many-body lattice sums}
\label{sec:zeta_representation}
We begin our discussion by introducing the concept of Bravais lattices. 
\begin{definition}[Lattices]
    Let $A\in \mathds R^{d\times d}$ regular. Then, we call $\Lambda = A\mathds Z^d$ a Bravais or monoatomic lattice. An important quantity here is the volume of the elementary lattice cell $V_\Lambda = \vert \det A\vert$. The reciprocal lattice is defined as $\Lambda^\ast = A^{-T} \mathds Z^d$. Finally, the Brillouin zone $\mathrm{BZ} = A^{-T}(-1/2,1/2)^d$ is the elementary lattice cell of the reciprocal lattice.
\end{definition}

The main ingredient for our method for computing many-body lattice sums is the Epstein zeta function. Originally introduced by Epstein in 1903, it generalizes the Riemann zeta function to oscillatory sums over lattices in higher dimensions \cite{epstein1903theorieI,epstein1903theorieII}. 

\begin{definition}[Epstein zeta function]
    For a lattice $\Lambda =A \mathds Z^d$, $\bm x,\bm k\in \mathds R^d$, and $\nu\in \mathds C$, we define the Epstein zeta function as follows
    \[
   \new{\zeps{\Lambda,\nu}{\bm x}{\bm k}}=
    \sums_{\bm z \in \Lambda} \frac{e^{-2\pi i \bm z\cdot \bm k}}{\vert \bm z-\bm x \vert^\nu},\quad \mathrm{Re}(\nu)>d,
    \]
    and meromorphically continue to $\nu\in \mathds  C$. We call the special case $\bm x=\bm 0$ the simple Epstein zeta function, which we denote by $Z_{\Lambda,\nu}(\bm k)$.
\end{definition}

Different methods for computing the Epstein zeta function and its meromorphic continuation have been suggested in the past, starting from early works by Chowla and Selberg \cite{chowla1949epstein}, over expansions in terms of Bessel functions by Terras \cite{terras1973bessel}, to an expansion in terms of incomplete gamma functions by Crandall \cite{crandall2012unified}. In our recent work \cite{buchheit2024epstein}, we have developed an efficient algorithm for computing the Epstein zeta function, based on Crandall's work, with a high-performance implementation in the open source library EpsteinLib, freely available on \href{https://github.com/epsteinlib/epsteinlib}{GitHub}. It also discusses the analytic properties of the Epstein zeta function in all its arguments. The library has already been used to successfully predict fractional magnetization in Ising compounds \cite{yadav2024observation}. It has further been applied to study long-range interacting hardcore bosons \cite{koziol2024quantum}, and the melting of Devil's staircases in quantum magnets \cite{koziol2025melting}. The Epstein zeta function forms the foundation for the Singular Euler--Maclaurin (SEM) expansion, a recent generalization of the Euler--Maclaurin summation formula to physically relevant long-range interactions on higher-dimensional lattices \cite{buchheit2022singular,buchheit2022efficient}.

We now consider the computation of the following many-body lattice sums, which we call $n$-body zeta functions.
\begin{definition}[Many-body zeta function]
\label{multi-body-zeta}
    Let $\Lambda=A\mathds Z^d$ with $A\in \mathds R^{d\times d}$ regular, $n\in \mathds N_+$, and $\bm \nu \in \mathds C^n$. We define the $n$-body zeta function as the lattice sum
    \[
    \zeta_\Lambda^{(n)}(\bm \nu) = \sideset{}{'}\sum_{\bm x^{(1)},\dots, \bm x^{(n-1)} \in \Lambda}~  \prod_{j=1}^{n} \frac{1}{\vert \bm x^{(j)}-\bm x^{(j-1)} \vert^{\nu_j}},\quad \mathrm{Re}(\new{\nu_j})>d,\quad \new{j=1,\dots,n},
    \]
    \new{with fixed $\bm x^{(0)}=\bm x^{(n)}\in \Lambda$}. \new{The choice of $\bm x^{(0)}$ does not alter the value of the sum, due to translational invariance, and we can thus set $\bm x^{(0)}=\bm 0$.}
\end{definition}
\new{The well-definedness of the many-body zeta function will be discussed in \Cref{thm:epstein_representation}.}

\begin{figure}
    \centering
    \includegraphics[width=1.0\linewidth]{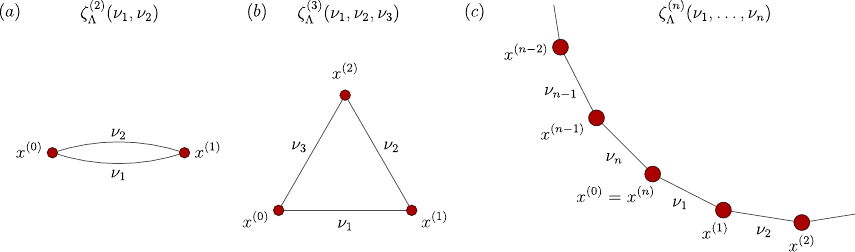}
    \caption{The many-body zeta function $\zeta^{(n)}_\Lambda$ described as a circular graph for (a) $n=2$, (b) $n=3$, and $(c)$ general $n$. }
    \label{fig:circular}
\end{figure}
These sums can be represented by circular graphs as shown in Fig.~\ref{fig:circular}, where the nodes correspond to lattice site and the edges describe the couplings with corresponding interaction exponents. The arising lattice sums appear directly in the evaluation of many-body lattice sums. The case $\nu_i=3$ corresponds to the isotropic part of the $n$-body interaction energy due to dipole-dipole van-der-Waals interactions in a Drude model \cite[Eqs.~(7b), (9)]{bade1957drude}. Related lattice sums in dimension $d=1$, sometimes denoted as multiple zeta functions, have been discussed in number theory \cite{gil2017multiple,zhao2016multiple}. The lattice sums in \Cref{multi-body-zeta} appear as a central ingredient for high-order perturbative treatments of quantum spin systems with long-range interactions, see e.g.~\cite[Eq.\,(179)]{adelhardt2024monte}.

An appropriate recombination of $3$-body zeta functions allows us to generate general power-law three-body interactions. This includes the highly relevant case of the Axilrod-Teller-Muto (ATM) cohesive energy, which we show in the following proposition.

\begin{proposition}
    Let $\Lambda = A\mathds Z^d$, with $A\in \mathds R^{d\times d}$ regular. Then, the ATM cohesive energy as in \Cref{def:atm} can be rewritten in terms of three-body zeta functions as follows,
    \[
    E^{(3)}_\mathrm{coh}
    = \frac{1}{24}\zeta_\Lambda^{(3)}(3,3,3) - \frac{3}{16} \zeta_\Lambda^{(3)}(-1,5,5)+\frac{3}{8} \zeta_\Lambda^{(3)}(1,3,5).
    \]
\end{proposition}
\begin{proof}
The first term in ATM potential in Def.~\ref{def:atm} is already in the form of a three-body zeta function and equals $ \zeta_{\Lambda}^{(3)}(3,3,3)/6$. 
    For the second, first notice that 
    \[
\bm{x}\cdot \bm{y} = (\bm{x}^2+\bm{y}^2-\bm{z}^2)/2,\quad
\bm{y}\cdot \bm{z} = (\bm{y}^2+\bm{z}^2-\bm{x}^2)/2,\quad
 \bm{z}\cdot \bm{x} =- (\bm{z}^2 + \bm{x}^2-\bm{y}^2)/2.
   \]
    \new{using the notation $\bm x^2=x_1^2+\dots+x_d^2$ and} $\bm z=\bm y-\bm x$, which we adopt from now on. We then find
    \[
    -\frac{1}{2}\,\sideset{}{'}\sum_{\bm x,\bm y\in \Lambda} \frac{(\bm x\cdot \bm y)(\bm y\cdot \bm z)(\bm z\cdot \bm x)}{\vert \bm x \vert^{5}\vert \bm y \vert^{5}\vert \bm  z\vert^{5}} = \frac{1}{16}\,\sideset{}{'}\sum_{\bm x,\bm y\in \Lambda} \frac{(\bm x^2+\bm y^2-\bm z^2)(\bm y^2+\bm z^2-\bm x^2)(\bm z^2+\bm x^2-\bm y^2)}{\vert \bm x \vert^{5}\vert \bm y \vert^{5}\vert \bm  z\vert^{5}}.
    \]
    Using that permutations of $\{\bm x,\bm y,\bm z\}$ do not alter the value of the lattice sum, the right-hand side equals
    \[
    \frac{1}{16}\,\sideset{}{'}\sum_{\bm x,\bm y\in \Lambda} \bigg(- 2 \frac{1}{\vert \bm x \vert^{3}\vert \bm y \vert^{3}\vert \bm  z\vert^{3}}-3 \frac{\vert \bm x\vert}{\vert \bm y \vert^{5}\vert \bm  z\vert^{5}} + 6 \frac{1}{\vert \bm x \vert^{1}\vert \bm y \vert^{3}\vert \bm  z\vert^{3} } \bigg).
    \]
    We then find the desired result after writing these sums in terms of three-body zeta functions and including the \new{first} term.
\end{proof}

\new{We will now derive an efficiently computable representation of the $n$-body zeta function in terms of integrals over products of Epstein zeta functions. For the proof, we \new{will use the following well-known result from Fourier analysis,}
    \[
    V_\Lambda \int_{\mathrm{BZ}} e^{-2\pi i \bm x\cdot \bm k}\,\mathrm d \bm k = \delta_{\bm x,\bm 0},
    \]
with $\delta$ the Kronecker delta, which is equivalent to the tensorized one-dimensional result after the substitution $\tilde{\bm k}=A^T \bm k$.}

We then present the main result of this work, the Epstein representation of $n$-body zeta functions. It allows to rewrite  $n$-body zeta functions in terms of integrals involving products of Epstein zeta functions. Using EpsteinLib and knowledge of the singularities of the Epstein zeta function, these integrals can be efficiently evaluated.
\begin{theorem}[Epstein representation]
\label{thm:epstein_representation}
    Let $\Lambda=A\mathds Z^d$ with $A\in \mathds R^{d\times d}$ regular and $\bm \nu \in \mathds C^n$, $n\in \mathds N_+$, such that $\Re{\nu_i}>d$ for $1\le i\le n$. The $n$-body zeta function is then well-defined and admits the representation 
    \[
    \zeta_\Lambda^{(n)}(\bm \nu)=V_\Lambda\, \int \limits_{\mathrm{BZ}} \prod_{i=1}^n Z_{\Lambda,\nu_i} (\bm k) \,\mathrm d \bm k.
    \]
    The right-hand side can be meromorphically continued to $\nu_i\in \mathds C$ by means of the Hadamard integral.
\end{theorem}
\begin{proof}
    \new{We begin with the defining lattice sum, \[\sideset{}{'}\sum_{\bm x^{(1)},\dots, \bm x^{(n-1)} \in \Lambda}~  \prod_{j=1}^{n} \frac{1}{\vert \bm x^{(j)}-\bm x^{(j-1)} \vert^{\nu_j}},\quad \mathrm{Re}(\new{\nu_j})>d,\quad \new{j=1,\dots,n},\] choosing the origin such that $\bm x^{(n)} = \bm x^{(0)} = \bm 0$ and first show well-definedness.} \new{This is achieved by} introducing several changes of variables, beginning with
    \[
    \bm x^{(1)}=\bm z^{(1)}, \quad \bm x^{(2)}= \bm z^{(1)}+\bm z^{(2)}.
    \]
    Noting that $\Lambda-\bm x^{(1)}=\Lambda$, we then have 
    \begin{align*}
    \zeta_\Lambda^{(n)}(\bm \nu) = \sideset{}{'}\sum_{\bm z^{(1)},\bm z^{(2)},\bm x^{(3)},\dots, \bm x^{(n-1)} \in \Lambda} &\vert \bm z^{(1)} \vert^{-\nu_1} \vert \bm z^{(2)} \vert^{-\nu_2} \big \vert \bm x^{(3)}-(\bm z^{(1)}+\bm z^{(2)}) \big\vert^{-\nu_3}\\ &\times \prod_{j=4}^{n}\vert \bm x^{(j)}-\bm x^{(j-1)} \vert^{-\nu_j}.
    \end{align*}
    After substituting
    \[
    \bm x^{(i)}= \sum_{j=1}^i\bm z^{(j)},\quad i=1,\dots, n-1,
    \]
    the lattice sum takes the following form
    \[
    \sideset{}{'}\sum_{\bm z^{(1)},\dots,\bm z^{(n-1)} \in \Lambda} \prod_{i=1}^{n-1}\vert \bm z^{(i)} \vert^{-\nu_i} \bigg\vert -\sum_{i=1}^{n-1} \bm z^{(i)}\bigg\vert^{-\nu_n}.
    \]
    This sum over $n-1$ lattice vectors can be enlarged to a sum of $n$ lattice vectors with the restriction that 
    \[
    \bm z^{(n)} = -\sum_{i=1}^{n-1} \bm z^{(i)},
    \]
    which is equivalent to the condition that the sum over all $\bm z^{(i)}$ vanishes.
    Thus
    \[
    \zeta_\Lambda^{(n)}(\bm \nu) = \sideset{}{'}\sum_{\bm z^{(1)},\dots,\bm z^{(n)} \in \Lambda} \prod_{i=1}^n \vert \bm z^{(i)}\vert^{-\nu_i}
     \delta_{\bm z^{(1)}+\ldots+\bm z^{(n)},\bm 0}
    \;.
    \]
    \new{Absolute convergence and thus well-definedness of the above sum under the condition $s_j=\mathrm{Re}(\nu_j)>d$ then follows directly from the bound
    \[
   |\zeta_\Lambda^{(n)}(\bm \nu)| \le \sideset{}{'}\sum_{\bm z^{(1)},\bm z^{(2)},\dots, \bm z^{(n)} \in \Lambda}\prod_{j=1}^{n} \vert \bm z^{(j)} \vert^{-s_j} =\prod_{j=1}^{n} Z_{\Lambda,s_j}(\bm 0)<\infty.\]} \new{We continue by transforming the equality condition on the right-hand side} into an integral over the Brillouin zone
    \[
     \sideset{}{'}\sum_{\bm z^{(1)},\dots,\bm z^{(n)} \in \Lambda} \prod_{i=1}^n \vert \bm z^{(i)}\vert^{-\nu_i} V_\Lambda \int_\mathrm{BZ} e^{-2\pi i \sum_{i=1}^n \bm z^{(i)}\cdot \bm k } \,\mathrm d\bm k.
    \]
    
    Due to absolute convergence of the sum, the Fubini-Tonelli theorem permits to exchange the order of integration and summation. We then pair the oscillatory terms with the matching algebraic singularities to obtain the desired products of Epstein zeta functions
    \[
    \zeta_\Lambda^{(n)}(\bm \nu) = V_\Lambda \int_{\mathrm{BZ}} \prod_{i=1}^n \bigg(\sum_{\bm z\in \Lambda} \frac{e^{-2\pi i \bm z\cdot \bm k}}{\vert \bm z\vert^{-\nu_i}}\bigg)\,\mathrm d \bm k = V_\Lambda \int_{\mathrm{BZ}} \prod_{i=1}^n Z_{\Lambda,\nu_i}(\bm k)\,\mathrm d \bm k.
    \]
    The meromorphic continuation by means of the Hadamard regularization \cite{gelfand1964generalizedI} is then possible as the Epstein zeta function can be separated into a homogeneous algebraic singularity and an analytic regularized function, see \cite{buchheit2024computation}.
\end{proof}

For small values of $n$, the many-body zeta function can be written in terms of known functions. 
While the $1$-body zeta function vanishes, the $2$-body zeta function again reduces to an Epstein zeta function, which is summarized in the following corollary.

\begin{corollary}
\label{OneTwoBodyZetaFunction}
    Let $\Lambda=A\mathds Z^d$ with $A\in \mathds R^{d\times d}$ regular, and $\nu_1,\nu_2\in \mathds C$. Then
    \begin{align*}
    \zeta_{\Lambda}^{(1)}(\nu_1) &= V_\Lambda \,\ddashint_{\mathrm{BZ}} Z_{\Lambda,\nu_1}(\bm k)\,\mathrm d \bm k = 0,\\
    \zeta_{\Lambda}^{(2)}(\nu_1,\nu_2) & = V_\Lambda \,\ddashint_{\mathrm{BZ}} Z_{\Lambda,\nu_1}(\bm k)Z_{\Lambda,\nu_2}(\bm k)\,\mathrm d \bm k = Z_{\Lambda,\nu_1+\nu_2}(\bm 0),
    \end{align*}
    where the dashed integral denotes the meromorphic continuation to $\nu_i\in \mathds C$ by means of the Hadamard integral.
    \begin{proof}
        \new{In analogy to the proof of \Cref{thm:epstein_representation}, we} can exchange integral and sum for $\mathrm{Re}(\nu_i)>d$ by the Fubini-Tonelli theorem. The integral over the exponential yields
        \[
        \zeta_{\Lambda}^{(1)}(\nu_1) = ~\sideset{}{'} \sum_{\bm x\in \Lambda} \frac{\delta_{\bm x,\bm 0}}{\vert \bm x\vert^\nu} = 0,
        \]
        where the right-hand side is, of course, already holomorphic for $\nu\in \mathds C$. Following an analogous argument, we find 
        \begin{align*}
          \zeta_{\Lambda}^{(2)}(\nu_1,\nu_2) =\,\sideset{}{'} \sum_{\bm x,\bm y\in \Lambda} \frac{\delta_{\bm x+\bm y,\bm 0}}{\vert \bm x\vert^{\nu_1}\vert \bm y\vert^{\nu_2}} = Z_{\Lambda,\nu_1+\nu_2}(\bm 0).
        \end{align*}
        \new{The result then follows by the uniqueness theorem for holomorphic functions.}
    \end{proof}
\end{corollary}

\section{Quadrature and meromorphic continuation}
\label{sec:num_int}
Special care needs to be taken in the numerical computation of the arising integrals, as the Epstein zeta function exhibits a power-law singularity at $\bm k=0$. One can show that the Epstein zeta function can be decomposed into the sum of an analytic function $Z_{\Lambda,\nu}^{\mathrm{reg}}(\bm k)$ on the Brillouin zone and a power-law or logarithmic singularity $\hat s_\nu(\bm k)=\mathcal F (\vert \bm \cdot \vert^\nu) $. The decomposition reads \cite{buchheit2024computation,buchheit2024epstein}
\[
Z_{\Lambda,\nu}(\bm k)=Z_{\Lambda,\nu}^{\mathrm{reg}}(\bm k) +\frac{\hat s_\nu(\bm k)}{\vert \det A\vert},
\]
and 
\begin{align*}
\hat s_\nu(\bm k) & = \pi^{\nu-d/2}\frac{ \Gamma ((d-\nu)/2)}{\Gamma(\nu/2) } |\bm k|^{\nu-d}, && \nu\neq d+2\mathds {N}_0, \\
\hat s_\nu(\bm k) & = \frac{\pi^{2m+d/2}}{\Gamma(m+d/2)}\frac{(-1)^{m+1}}{m!} \new{|\bm k|^{2m} \log (\pi  \bm k^{2})},&& \nu= d+2m,\ m\in \mathds N_0,
\end{align*}
\new{where we again employ the notation $\bm k^2=k_1^2+\ldots+k_d^2$.}
\new{
Note that for $\nu\in d+2\mathds N_0$, an additional logarithmic factor appears, see i.e. \cite[Sec. 2.4]{buchheit2024epstein}.
}
The integration of the smooth part of the integral is then performed via Gauss quadrature, whereas the singularity  is handled via a Duffy transformation \cite{duffy1982quadrature}.
\begin{corollary}[Duffy transformation]
\label{duffy}
    For $\Lambda=A\mathds Z^d$ and $\bm \nu\in \mathds C^n$ where $\mathrm{Re}(\nu_i)>d$, $1\le i\le n$, we have
    $$
    \zeta_\Lambda^{(n)}(\bm \nu) = \new{2}\sum_{\bm p\in \{\pm 1\}^{d - 1} }\sum_{j=0}^{d-1} \int_{0}^{1/2} u^{d-1} \int_{\new{[0,1]}^{d-1}}   f\big(u \new{\bm w (\bm v)}\big)\,\mathrm d \bm v\,\mathrm d u,
    $$
    with 
    $$\bm w(\bm v) =A^{-T}\big(\sigma^j (\new{\bm p \bm v}, 1)^T\big) =A^{-T} \big(\sigma^j (\new{p_1 v_1},\dots ,\new{p_{d-1} v_{d-1}} , 1)^T\big),$$ 
    the cyclic permutation $\sigma (x_1,\dots, x_d)^T = (x_d,x_1,\dots ,x_{d-1})^T$ and
    $$
    f(\bm k) = \prod_{i=1}^n Z_{\Lambda,\nu_i}(\bm k).
    $$
\end{corollary}
\begin{proof}
    We first use that the Epstein zeta function is an even function in $\bm k$ to reduce the integration to half the Brillouin zone. After the substitution $\tilde {\bm k} = A^T\bm k$ and leaving away the tilde, we then have
    $$
    \zeta_{\Lambda}^{(n)}(\bm \nu)=2 \sum_{\bm p\in \{\pm 1\}^{d - 1} } \int_{[0,1/2]^d} f(A^{-T}(p_1 k_1, \dots ,p_{d-1} k_{d-1}, k_d)^T )\,\mathrm d \bm k,
    $$
    with $f$ defined as above.
     Now separate the integrals over the corner $[0,1/2]^d$ into $d$ integrals over pyramids,
    $$
    \int_{[0,1/2]^d} g(\bm k)\,\mathrm d \bm k =\sum_{j=0}^{d-1} \int_{0}^{1/2} \int_{[0,k_d]^{d-1}} g(\sigma^j (k_1,\dots,k_{d-1},k_d)^T) \,\mathrm d  k_1\dots  d k_{d-1} \,\mathrm d k_d.
    $$
    Finally, applying the Duffy transform
    $$
    k_j = \new{u v_j},\quad {j=1,\dots,d-1}, \quad k_d = u,
    $$
and letting $g(\bm k)=f(A^{-T}(p_1 k_1, \dots ,p_{d-1} k_{d-1}, k_d)^T )$ yields
\[
    \int_{[0,1/2]^d} g(\bm k)\,\mathrm d \bm k =\sum_{j=0}^{d-1} {\int_0^{1/2}}\int_{{\new{[0,1]}^{d-1}}}  {u^{d-1}} f(u\, A^{-T}\sigma^j ( \new{\bm p\bm v}, 1)^T )\,\mathrm d \bm v\,\mathrm d u.
 \qedhere
\]
\end{proof}

\new{In the following, we focus on real exponents $\nu_j$.}
\new{The integral over $\bm v$ in \Cref{duffy} is computed through an $n_{\rm q}$-point Gauss-Legendre quadrature tensorized to $d-1$ dimensions, while the one-dimensional singular integral in $u$, for $\nu_j \ge d$, can be efficiently computed by adaptive integration. The efficient computation of the meromorphic continuation, used if at least one $\nu_j< d$, is discussed at the end of this section.}

\new{We now investigate how the quadrature error scales with the number of quadrature points depending on the condition number of the lattice matrix $A$. To this end, we first show that the integrand in the preceding corollary admits a holomorphic extension with respect to its argument $\bm v$ to a complex neighborhood of the integration domain. The argument relies on a recent result establishing the holomorphic extension of the Epstein zeta function.
We then derive a bound for the largest Bernstein ellipse in each integration dimension, with foci at the endpoints of the integration interval, that is contained in the region of holomorphy. By standard results, this implies exponential convergence of the $L^\infty$-approximation error of $f$ by orthogonal polynomials. Consequently, the quadrature error of a tensor-product $n_{\rm q}$-point Gauss--Legendre rule decays exponentially with respect to the number of quadrature points, see i.e.~\cite[Thm.~5.3.15]{sauter2010boundary}.}

\new{While the holomorphic extension of the Epstein zeta function in its argument $\nu$ had already been discussed by Epstein \cite{epstein1903theorieI,epstein1903theorieII}, the holomorphic extension with respect to its remaining arguments $\bm x$ and $\bm k$ was, until recently, not well understood.
In \cite{buchheit2024epstein}, we show that the Epstein zeta function $Z_{\Lambda,\nu}(\bm x,\bm k)$  admits a jointly holomorphic extension in $(\nu,\bm x,\bm k)$. The following theorem is a weaker version of \cite[Th.~2.13, Th.~2.16]{buchheit2024epstein} that focuses on the holomorphy of the Epstein zeta function and of the regularized Epstein zeta function with respect to the argument $\bm k$ for $\bm x=\bm 0$, which is required for the analysis of the quadrature error. We here replace $\bm k$ by $\bm q \in \mathds C^d$ to signal the holomorphic continuation.
}

\new{
\begin{theorem}
\label{holreg}
Let $\Lambda$ be a $d$-dimensional lattice and $\nu\in\mathds C$. 
Denote by 
$$
D_L=\{\bm q\in\mathds C^d:|\Re{\bm q}-\bm z|>|\Im{\bm q}|\ \forall \bm z\in L\}.
$$
the intersection of $d$-dimensional complex cones with origins at $L\subseteq \mathds R^d$.
Then, $Z_{\Lambda,\nu}(\bm \cdot)$ can be holomorphically extended to $D_{\Lambda^\ast}$
and $Z^{\rm reg}_{\Lambda,\nu}(\bm \cdot)$ 
can be holomorphically extended to 
$D_{\Lambda^*\setminus\{\bm 0\}}$.
\end{theorem}
}
\new{
By applying the Duffy transformation, we transform the $d$-dimensional integral over $\bm k\in \mathrm{BZ}$ with a power-law singularity at $\bm k=\bm 0$ into a sum of one-dimensional integrals in $u$, each with power-law singularity at $u=0$, and $(d-1)$-dimensional smooth integrals over $\bm v\in [0,1]^{d-1}$.
The following corollary leverages \Cref{holreg} to show that the integrand in $\bm v$ is not only smooth on the integration domain $[0,1]^{d-1}$, but even extends to a holomorphic function in a complex neighborhood of $\bm v \in \mathds R^{d-1}$.
}

\new{
\begin{corollary}
\label{hol-integrand}
Let $0<u<1/2$ be fixed. Then the function $f(u\bm w(\bm v))$, as defined in \Cref{duffy}, admits a holomorphic extension in $\bm v$ to the domain
\[
\Omega=\left\{\bm v\in \mathds C^{d-1} : |\Im{\bm v}| < 1/\kappa(A)\right\}.
\]
Here, $\kappa(A)=\sigma_\mathrm{max}/\sigma_\mathrm{min}$ is the condition number of $A$ with
$\sigma_{\rm min}$ and $\sigma_{\rm max}$ denoting the smallest and largest singular values of $A$.
\end{corollary}
}
\begin{proof}
\new{Recall that \(f\) is a product of Epstein zeta functions and that
\[
\bm w(\bm v)
=
A^{-T}\bigl(\sigma^j(p_1v_1,\dots,p_{d-1}v_{d-1},1)^T\bigr),
\]
where \(\sigma\) denotes a cyclic permutation and \(p_i\in\{\pm1\}\). We begin with the holomorphic extension of \(f\) itself and then incorporate, step by step, the permutation, sign changes, rescaling, and dimension reduction.}

\new{The Epstein zeta function $Z_{\Lambda,\nu_i}$ extends holomorphically to $D_{\Lambda^*}$ by \Cref{holreg}. Thus, $f$ extends to a holomorphic function on $D_{\Lambda^*}$
as the product of holomorphic functions and,
by the composition theorem,
$f(A^{-T} \bm q)$ is holomorphic on 
$\bm q\in 
A^TD_{\Lambda^*}$. As $\Lambda^\ast=A^{-T} \mathds Z^d$, we find after inserting in \Cref{holreg} that
\[
A^T D_{\Lambda^*}=\{\bm q\in\mathds C^d:|A^{-T}(\Re{\bm q}-\bm z)|>|A^{-T}\Im{\bm q}|\ \forall \bm z\in \mathds Z^d\}.
\]}
\new{Consider the simplified set
\[
K
=
\{\bm q\in\mathds C^d:\ |\Re{\bm q}-\bm z|>\kappa(A)|\Im{\bm q}|
\quad \forall \bm z\in\mathds Z^d\},
\]
which depends only on the condition number \(\kappa(A)\). We show that $K\subset A^T D_{\Lambda^\ast}$.
For \(\bm q\in K\) and \(\bm z\in\mathds Z^d\), the singular-value bounds
\[
|A^{-T}\bm k|\ge \sigma_{\min}(A^{-T})|\bm k|,
\qquad
|A^{-T}\bm k|\le \sigma_{\max}(A^{-T})|\bm k|,
\qquad \bm k\in\mathds R^d,
\]
together with
$
\sigma_{\min}(A^{-T})=1/\sigma_{\max}(A)$, and
$
\sigma_{\max}(A^{-T})=1/\sigma_{\min}(A),
$
yield
\begin{align*}
|A^{-T}(\Re{\bm q}-\bm z)|
\ge \frac{1}{\sigma_{\max}(A)}|\Re{\bm q}-\bm z|
> \frac{\kappa(A)}{\sigma_{\max}(A)}|\Im{\bm q}|
&= \frac{1}{\sigma_{\min}(A)}|\Im{\bm q}|\\
&\ge |A^{-T}\Im{\bm q}|.
\end{align*}
Hence, we obtain \(\bm q\in A^T D_{\Lambda^\ast}\), proving that \(K\subset A^T D_{\Lambda^\ast}\). Therefore, \(f(A^{-T}\bm q)\) is holomorphic on \(\bm q\in K\).
Permutations and sign change can further be included via the signed permutation matrix \[Q_{j,\bm p}=\sigma^j \mathrm{diag}(p_1,\dots,p_{d-1},1),\] with $p_i= \pm 1$. Then $f(A^{-T} Q_{j,\bm p}\bm q)$ is also holomorphic on $K$, as the set inherits invariance under entry-wise sign changes and permutations of vector entries from the relation $Q_{j,\bm p}\mathds Z^d =\mathds Z^d$. Moreover, for $0<u<1/2$, we have that $f(u A^{-T} Q_{j,\bm p}\bm q)$ is holomorphic on $K/u$.}

\new{Finally, we transfer this result to holomorphy of $f(u\bm w(\bm v))$ in $\bm v\in\mathds C^{d-1}$. By the preceding argument, $f(u\bm w(\bm v))$ is holomorphic in $\bm v$ on
\[
\left\{
\bm v\in\mathds C^{d-1}:
\left|\Re{u(\bm v,1)^T}-\bm z\right|
>
\kappa(A)\left|\Im{u(\bm v,1)^T}\right|
\quad \forall \bm z\in\mathds Z^d
\right\}.
\]
Since $u>0$, squaring the inequality and separating the last component shows that the domain is equal to
\[
\left\{
\bm v\in\mathds C^{d-1}:
\left|\Re{\bm v}-\tilde{\bm z}/u\right|^2
+
\left(1-z_d/u\right)^2
>
\kappa(A)^2 |\Im{\bm v}|^2
\quad
\forall\, \tilde{\bm z}\in\mathds Z^{d-1},\ z_d\in\mathds Z
\right\}.
\]
We now have that $\left|\Re{\bm v}-\tilde{\bm z}/{u}\right|^2 \ge 0$,
and, since $0<u<1/2$,
\[
\left(1-z_d/u\right)^2 \ge 1
\qquad \forall\, z_d\in\mathds Z,
\]
where the bound is reached at $z_d=0$.
Therefore, the above domain contains
\[
\Omega=\left\{
\bm v\in\mathds C^{d-1}:
|\Im{\bm v}|<1/\kappa(A)
\right\}.
\]
on which $f(u\bm w(\bm v))$ is therefore holomorphic as stated.}
\end{proof}
\new{In the next step, we determine the largest Bernstein ellipse contained in the domain of holomorphy, beginning with its definition, see i.e. \cite[Chap.~5.3.2.2]{sauter2010boundary}.}%
\new{\begin{definition}
For $\rho>1$, we define the Bernstein ellipse $E_\rho$ with foci at $0$ and $1$ as
\[
E_{\rho}
=
\left\{
v\in\mathbb{C} :
v=\frac12+\frac{\xi+\xi^{-1}}{4},
\ \xi\in\mathbb{C},\ |\xi|=\rho
\right\},
\]
where the semi-major and semi-minor axes have lengths $(\rho+\rho^{-1})/4$ and $(\rho-\rho^{-1})/4$. Equivalently, $\rho$ is the sum of these lengths normalized by the radius $1/2$ of the interval $[0,1]$.
\end{definition}}%
\new{Here, we have slightly modified the notation from \cite{sauter2010boundary} by rescaling $\rho$ with the interval radius, which simplifies the subsequent error discussion. We then determine the largest Bernstein ellipse contained in the domain of holomorphy.}
\new{\begin{corollary}
\label{hol-bernstein}
Let $0<u<1/2$ be fixed and $d>1$.
Then
$f(u\bm w(\bm v))$ as in \Cref{duffy} extends to a holomorphic function in $\bm v$ on
$E_{\rho}^{d-1}=E_{\rho}\times\ldots\times E_{\rho}$
for all $1<\rho<\rho_\mathrm{bound}$,
where
\[
\rho_\mathrm{bound}=a+\sqrt{1+a^2},\quad a=\frac{2}{\kappa(A)\sqrt{d-1}}.
\]
\end{corollary}
\begin{proof}
By \Cref{hol-integrand}, $f(u \bm w(\bm v))$ extends holomorphically to $\bm v\in \Omega$. It therefore suffices to show that $E_\rho^{d-1}\subset \Omega$. We have 
\[
\vert \Im{\bm v}\vert \le \sqrt{d-1} \max_{1\le i\le d-1 } \vert \Im{v_i}\vert \le  \sqrt{d-1} \frac{\rho-\rho^{-1}}{4},
\]
as the semi-minor axis of $E_\rho$ has length $(\rho-\rho^{-1})/4$. From this follows the sufficient condition
\[
\rho-\rho^{-1} < \frac{4}{\kappa(A)\sqrt{d-1}}=2a.
\]
As $\rho>1$, this is equivalent to $\rho<a+\sqrt{1+a^2}$.
\end{proof}}

\new{
It follows that the error of a tensorized $n_{\rm q}$-point Gauss--Legendre quadrature rule applied to the $(d-1)$-dimensional integral
\[
I=\int_{[0,1]^{d-1}}   f(u \bm w (\bm v))\,\mathrm d \bm v
\]
 falls of geometrically as $\rho^{-2n_{ q}}$
for any $1<\rho<\rho_\mathrm{bound}$ for $\rho_\mathrm{bound}$ as in \Cref{hol-bernstein}, see for instance
\cite[Th. 5.3.15]{sauter2010boundary}.}

\new{We briefly comment on the structure of the one-dimensional integral in $u$ for $\nu_i\ge d$, which makes the effectiveness of a standard adaptive integration strategy clear, and then move on to the case of the meromorphic continuation. After separating the Epstein zeta function into its singularity and the analytic remainder and inserting this into the representation after Duffy transformation from \Cref{duffy}, the integral reads
\[ 
    \int_{0}^{1/2} u^{d-1}  f(u \bm w) \,\mathrm d u =  \int_{0}^{1/2} u^{d-1} \prod_{j=1}^n \left(c_{\nu_j} u^{\nu_j-d} \frac{\vert \bm w \vert^{\nu_j-d}}{\vert \det{A} \vert} + Z^\mathrm{reg}_{\Lambda,\nu_j}(u\bm w) \right)  \,\mathrm d u,
\]
where $u^{\nu_j-d}$ is replaced by $u^{\nu_j-d}\log(\pi u^2\bm w^2)$ for $\nu_j \in d+2\mathds N_0$. By \Cref{holreg}, the regularized Epstein zeta function is holomorphic in neighborhood of the origin and therefore $Z^\mathrm{reg}_{\Lambda,\nu_j}(u\bm w)$ is holomorphic in an environment of $u\in [0,1/2]$. After multiplying out the integrand, the integral can be separated into a sum of of one-dimensional integrals  of the form
\[
\int_{0}^{1/2}u^{\mu-1} \log(\pi u^2\bm w^2)^m g(u)\,\mathrm d u, \quad m\in \mathds N_0,
\]
with $g$ holomorphic in an environment of the integration domain and $\mu\ge d$. Note that higher powers of the logarithm can appear if multiple exponents $\nu_j \in d+2\mathds N_0$.
Each integrand is therefore the product of a holomorphic function in a neighborhood of the integration interval and an algebraic and/or logarithmic singularity at $u=0$. We therefore evaluate the integral using standard adaptive quadrature with recursive interval subdivision, which automatically concentrates subintervals near $u=0$ and thereby resolves all endpoint singularities simultaneously, see, i.e., \cite{piessens2012quadpack} and references therein. We note here that higher powers of the logarithm do not lead to a noticeable increase in the number of required bisections, due to the mildness of the singularity.
}

\new{The case of the meromorphic continuation, used conservatively if at least one $\nu_j < d$, is treated as follows.} We first divide the integration interval in $u$ into  $(0,\varepsilon)$ and $(\varepsilon,1/2)$. An adaptive integrator is then used to handle the second integral, \new{while the first is computed analytically, as discussed in the following}. For the meromorphic continuation of the integral over $(0,\varepsilon)$, we expand the regularized Epstein zeta function in an analytic Taylor series to sufficiently high order, yielding
\[
Z^\mathrm{reg}_{\Lambda,\nu_j}(u\bm w) = \sum_{k=0}^\ell u^{2k} C_{k,\nu_j}(\bm w) +\mathcal O (u^{2(\ell+1)}),
\]
\new{using that the regularized Epstein zeta function is even, with the derivatives}
\new{\[C_{k,\nu}(\bm w) =\frac{1}{(2k)!}(\bm w\cdot \nabla)^{2k} Z^{\mathrm{reg}}_{\Lambda,\nu}(\bm 0).
\]}

\new{
From \Cref{holreg}, we know that the regularized Epstein zeta function $Z^\mathrm{reg}_{\Lambda,\nu}$ extends to a holomorphic function around the origin with radius of convergence \[r=\mathrm{dist}(\bm 0,\Lambda^\ast\setminus\{\bm 0\})\ge \frac{1}{\sigma_\mathrm{max}(A)}.\]  On every compact subset of $B_r(\bm 0)$, the truncated Taylor series about the origin then converges uniformly to $Z^\mathrm{reg}_{\Lambda,\nu}$, and the truncation error decays exponentially in $\ell$. Here, we choose the splitting parameter $\varepsilon$ sufficiently large, such that cancellation error between the two contributions is sufficiently limited, and small enough, such that the uniform error of the Taylor series on $[0,\varepsilon]$ is smaller than machine precision.
}
\new{For $\mu>0$, the integrals are finally evaluated analytically as}
\new{\begin{align*}
\int_{0}^\varepsilon u^{\mu-1} &=\frac{\varepsilon^{\mu}}{\mu},\\ 
\int_{0}^\varepsilon u^{\mu-1} \log(\pi u^2{\bm w}^2) \,\mathrm d u &= -\epsilon^{\mu} \frac{2-\mu\log(\pi \varepsilon^2{\bm w}^2)}{\mu^2},\\
\int_{0}^\varepsilon u^{\mu-1} \log(\pi u^2{\bm w}^2)^2 \,\mathrm d u &= \epsilon^{\mu} \frac{8-\mu\log(\pi \varepsilon^2{\bm w}^2)(4 -\mu\log(\pi \varepsilon^2{\bm w}^2))}{\mu^3},\\
\int_{0}^\varepsilon u^{\mu-1} \log(\pi u^2{\bm w}^2)^3 \,\mathrm d u &= -\epsilon^{\mu} \frac{48 -\mu \log(\pi \epsilon^2{\bm w}^2) \big(24 -\mu\log(\pi \epsilon^2{\bm w}^2)(6 -\mu \log(\pi \epsilon^2{\bm w}^2))\big)}{\mu^4}
\end{align*}
and so forth, where the right-hand side forms the meromorphic continuation to $\mu\in \mathds C\setminus\{0\}$.}
\new{
Here, the powers of logarithm that appear when expanding the product for $\nu\in d+2\mathds N_0$ do not introduce an additional error in the quadrature in $\bm v$, as they are holomorphic as a function of $\bm v$ in the same holomorphy region as the full integrand in \Cref{hol-integrand} and \Cref{hol-bernstein}.
}
We precompute the partial derivatives of the Epstein zeta function up to order $\ell$ analytically using  Crandall's formula in \cite{buchheit2024computation,buchheit2024epstein} and obtain the directional derivatives via
\[
(\bm w\cdot \nabla)^{2k}Z^{\mathrm{reg}}_{\Lambda,\nu_j}(\bm 0) = \sum_{\substack{\vert \bm \alpha\vert =2k}} \binom{2k
}{\bm \alpha} {\bm w}^{\bm \alpha} { \nabla}^{\bm \alpha} Z^{\mathrm{reg}}_{\Lambda,\nu_j}(\bm 0)
\]
with the multinomial coefficient
\[
\binom{\ell 
}{\bm \alpha} = \frac{\ell!}{\alpha_1! \alpha_2! \dots \alpha_d !}.
\]

\new{Finally, given $\mu\in \mathds R$ and an even function $g$ with radius of analyticity around zero larger than $\varepsilon$, the Hadamard integral on $(0,\varepsilon)$ is evaluated as
\[
\ddashint_{0}^\varepsilon u^{\mu-1} g(u) = \sum_{k=0}^\ell \frac{\varepsilon^{\mu+2k}}{\mu+2k}\frac{g^{(2k)}(0)}{(2k)!} + \mathcal{O}(\varepsilon^{\mu+2(\ell+1)}),
\]
for $\mu+2k\neq 0$ for $k=0,\dots,\ell+1$. For a discussion of the logarithmic corrections in case that $\mu+2k=0$ arises, see \cite{buchheit2022singular,gelfand1964generalizedI}.
The error thus decreases with $\varepsilon$ under the condition
\[
\mu+2(\ell+1)>0,
\]
which is equivalent to
\[
\ell\ge \left\lfloor -\mu/2 \right\rfloor,
\]
with $\lfloor x \rfloor$ the greatest integer less than or equal to $x\in \mathds R$. This relation determines the lower bound for the truncation parameter $\ell$. As $0<\varepsilon \le 1/2$, the error decreases exponentially in $\ell$.
}


\section{Numerical results}
\label{sec:benchmarks}
In this section, we \new{first} analyze the precision of our approach for computing many-body zeta functions 
\new{with respect to the discretization parameters, namely the number of quadrature points and the truncation parameter of the Taylor series.
We subsequently verify our method for a wide range of arguments such as the number of bodies, multiple dimensions, exponents and geometries. Here,} we adopt the following test strategy. We first verify our approach for one- and two-body zeta functions, where we can determine analytical benchmarks. We subsequently compare our approach against direct summation for particular cases such as \new{low spatial dimension $d$} or large values of $\nu_i$, where a sufficiently precise reference value can still be obtained. We finally study the behavior of the $n$-body zeta function as $n$, and thus the sum dimension increases and analyze the runtime and precision of our method in this case.

\new{
\subsection{Error scaling with numerical discretization parameters}
}

\new{
The main numerical discretization parameter that determines the precision of our method is the number of quadrature points $n_{\rm q}$ per dimension in the tensorized Gauss-Legendre rule for evaluating the integral in $\bm v$ in \Cref{duffy}.
In addition, in case of the meromorphic continuation,
the cutoff parameter of the Taylor series $\ell$ needs to be considered.
At the end of the preceding section, we haved demonstrated geometric convergence of the method in these discretization parameters, which we verify numerically in what follows.}

\new{Beyond this, there are implementation choices of lesser importance. One concerns the adaptive integrator for the one-dimensional integral in $u$. Due to the standard structure of this integral, discussed in the previous section, this part is handled robustly by standard numerical integration routines, see, i.e., \cite{piessens2012quadpack}. A second choice arises in the meromorphic continuation, which is used whenever at least one $\nu_j < d$, namely the splitting parameter $\varepsilon$. Its choice, however, is closely connected to the Taylor cutoff $\ell$. The Taylor expansion must approximate the integrand uniformly to full precision on $[0,\varepsilon]$, while $\varepsilon$ should at the same time be large enough to limit cancellation errors between the two resulting integral contributions. In the numerical experiments below, we use $\varepsilon = 10^{-2}$, which proved adequate in all cases considered.
}

\begin{figure}
\centering\includegraphics[width=0.65\linewidth]{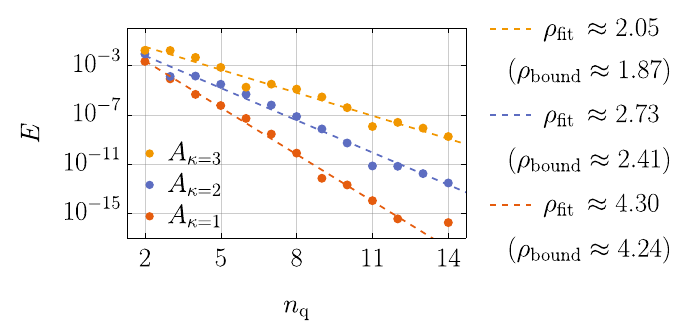}
    \caption{\new{Minimum of absolute and relative error of the three-body zeta function $\zeta^{(3)}_{A_{\kappa} \mathds Z^2}(1,3,5)$ with the lattice matrix given in \eqref{eq:A_kappa} and condition numbers $\kappa\in\{1,2,3\}$ (red, blue, yellow points) as a function of the number of quadrature points per dimension of a tensorized Gauss-Legendre rule in $\bm v$, where the value for $n_{\rm q}=20$ is taken as the reference. Geometric convergence in the number of quadrature points $n_{\rm q}$ is observed.
    The dashed lines display the fitted function $C\rho_{\rm fit}^{-2 n_{\rm q}}$. The corresponding values of the values of $\rho_\mathrm{fit}$ and the analytic bound $\rho_\mathrm{bound}$ are provided on the right.}
    }
    \label{fig:gauss-cond}
\end{figure}

\new{
As our measure of error $E$, we use the minimum of the absolute and relative error,
\[
E=\min(E_\mathrm{abs},E_\mathrm{rel}).
\]
We first study how the error decays with the number of quadrature points $n_{\rm q}$ per dimension in the $\bm v$-integral, and how this decay depends on the condition number of $A$. To this end, we consider a two-dimensional lattice matrix $A_\kappa$ 
\begin{equation}
\label{eq:A_kappa}
A_{\kappa}=
\begin{pmatrix}
1 & (\kappa - \kappa^{-1})/2 \\
0 & (\kappa+ \kappa^{-1})/2
\end{pmatrix},
\end{equation}
where the parametrization is chosen such that $\kappa \ge 1$ coincides with the condition number of the matrix. Here, $\kappa = 1$ reduces to the case of a two-dimensional square lattice. 
}

\new{
In \Cref{fig:gauss-cond} we display the minimum of absolute and relative error of the three-body zeta function 
$\zeta^{(3)}_{A_\kappa \mathds Z^2}(1,3,5)$ for $\kappa = 1,2,3$ (red, blue, yellow points)
as a function of the number of Gauss-Legendre quadrature points $n_{\rm q}$ per dimension. Here, we choose  $n_{\rm q}=20$ as our reference.
We observe geometric convergence   in the number of quadrature points as $\rho_\mathrm{fit}^{-2n_{\rm q}}$ with $\rho_\mathrm{fit}$ similar ($\kappa=1$) or slightly larger ($\kappa=2,3$) than the analytic bound for the rate $\rho_\mathrm{bound}$ obtained from \Cref{hol-bernstein}.
}

\new{
In \Cref{fig:gauss}
we investigate the error scaling for the three-body zeta function $\zeta^{(3)}_{\Lambda}(1,3,5)$ as a function of the number of quadrature points $n_{\rm q}$ for physically relevant lattices in two and three dimensions. In panel (a), we display the error for the two-dimensional  
 square lattice (blue) with 
$\Lambda_{\rm sq}=\mathds Z^2$ and
the hexagonal lattice (red)
$\Lambda_{\rm hex}=A_{\rm hex}\mathds Z^2$
with lattice matrix
$$
A_{\rm hex}=
\begin{pmatrix}
    1&  1/2 \\ 0 & \sqrt3/2
\end{pmatrix}.
$$
In panel (b), we display the equivalent error for the three-dimensional face-centered-cuboid lattice $\Lambda_{\rm fcc}=A_{\rm fcc}\mathds Z^3$ (red)
and the  body-centered-cuboid lattice $\Lambda_{\rm bcc}=A_{\rm bcc}\mathds Z^3$ (blue), given by the lattice matrices \cite{roblesnavarro2025exact}
$$
A_{\rm fcc}=
\frac{1}{\sqrt{2}}
\begin{pmatrix}
1 & 1 & 0 \\ 
1 & 0 & 1 \\ 
0 & 1 & 1
\end{pmatrix}
,\qquad
A_{\rm bcc}=
\frac{1}{\sqrt{3}}
\begin{pmatrix}
1 & 1 & 0 \\ 
\sqrt{2} & 0 & \sqrt{2} \\ 
0 & \sqrt{2} &\sqrt{2}
\end{pmatrix}.
$$
We observe an error $E\le 2\cdot 10^{-14}$ for $n_{\rm q}=14$ quadrature points, and adopt this value in all subsequent benchmarks.
}

\begin{figure}
    \centering
    \includegraphics[width=0.74\linewidth]{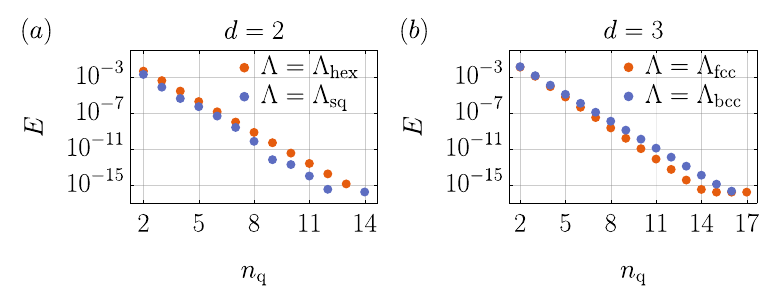}
    \caption{\new{Minimum of absolute and relative error of the three-body zeta function $\zeta^{(3)}_\Lambda(1,3,5)$ for $\Lambda\in\{\Lambda_{\rm sq},\Lambda_{\rm hex}\}$ in $d=2$ dimensions in (a) and $\Lambda\in\{\Lambda_{\rm fcc},\Lambda_{\rm bcc}\}$ in $d=3$ dimensions in (b) as a function of the number of quadrature points of the Gauss-Legendre quadrature per dimension in the integral in $\bm v$, where the numerical result for $n_{\rm q}=20$ is taken as the reference.
    We observe exponential convergence in the number of quadrature points and achieve a precision of $E \le 2\cdot10^{-14}$ with $n_{\rm q}= 14$ for all test cases.}}
    \label{fig:gauss}
\end{figure}

\new{
We finally analyze the error for the meromorphic continuation, which we use if at least one $\nu_j<d$. In \Cref{tab:taylor}, we display the error of the three-body zeta function as function of the truncation parameter $\ell$ of the Taylor series  for the square lattice and the hexagonal lattice  for representative exponents $(1,3,5)^T$ and $(-1,5,5)^T$ choosing $\varepsilon = 10^{-2}$. This corresponds to computing a Taylor series of order $2\ell$, as the function is even. We observe an exponential decay of the error in $\ell$. The error is stably below machine precision for $\ell\ge 3$, and we conservatively use $\ell=4$ in the following. Note that this choice respects the bound derived in the preceding section for all following examples, where $\nu_j\ge -3$.}

\begin{table}[ht]
\centering
\begin{tabular}{ccccccc}
\toprule
$\Lambda$ & $\nu$ 
& $\ell = 0$ 
& $\ell = 1$ 
& $\ell = 2$ 
& $\ell = 3$ \\

\midrule\multirow{3}{*}{$\Lambda_{\rm sq}$}
& $(1,3,5)$ &$8.34\cdot 10^{-5}$ & $1.08\cdot 10^{-8}$ & $1.57\cdot 10^{-13}$ & $0$ \\


& $(-1,5,5)$ &$4.23\cdot 10^{-2}$ & $1.95\cdot 10^{-6}$ & $2.00\cdot 10^{-11}$ & $1.42\cdot 10^{-15}$ \\

\midrule\multirow{3}{*}{$\Lambda_{\rm hex}$}
& $(1,3,5)$ &$9.19\cdot 10^{-5}$ & $1.38\cdot 10^{-8}$ & $4.96\cdot 10^{-14}$ & $0$ \\


& $(-1,5,5)$ &$4.25\cdot 10^{-2}$ & $2.06\cdot 10^{-6}$ & $4.88\cdot 10^{-12}$ & $1.67\cdot 10^{-15}$ \\
  
\bottomrule
\end{tabular}
    \caption{\new{Minimum of absolute and relative error of the three-body zeta function $\zeta^{(3)}_\Lambda(1,3,5)$ and $\zeta^{(3)}_\Lambda(-1,5,5)$ for $\Lambda\in\{\Lambda_{\rm sq},\Lambda_{\rm hex}\}$ in $d=2$ dimensions as a function of the Taylor-series cutoff $\ell$, corresponding to an order $2\ell$ expansion, with $\varepsilon=10^{-2}$, where the value for $\ell=6$ is taken as a reference.
    We observe exponential convergence in $\ell$. }
    }
\label{tab:taylor}
\end{table}

\subsection{One- and two-body zeta functions}

\begin{figure}
    \centering
    \includegraphics[width=1.
\linewidth]{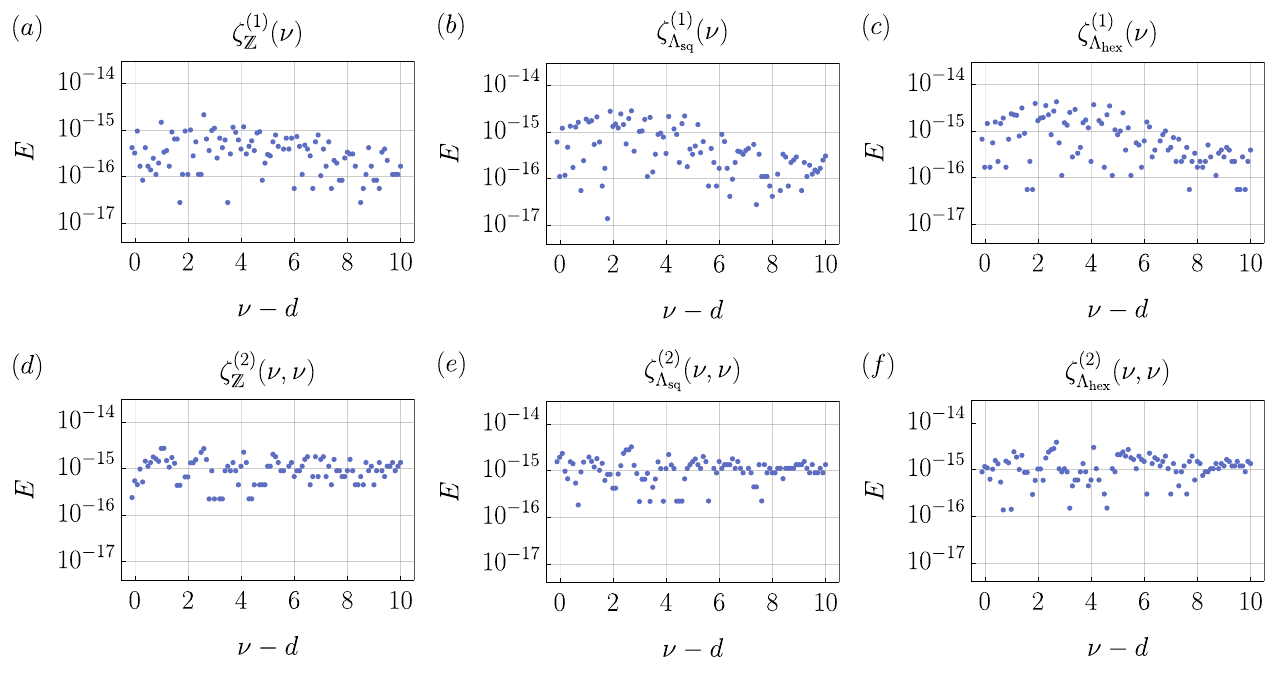}
    \caption{Minimum of absolute and relative error $E$ of the many-body zeta function $\zeta^{(n)}_\Lambda(\bm \nu)$ for $n\in\{1,2\}$ for different one- and two dimensional lattices $\Lambda\in\{\mathds Z,\Lambda_{\rm sq},\Lambda_{\rm hex}\}$.}
    \label{fig:OneTwoBodyZetaError}
\end{figure}

We benchmark our numerical implementation of the many-body zeta function
 $\zeta^{(n)}_\Lambda(\bm \nu)$ for $n=1$ and $n=2$ for the integer lattice $\Lambda=\mathds Z$, \new{the square lattice $\Lambda_{\rm sq}=\mathds Z^2$, and the hexagonal lattice $\Lambda_{\rm hex}=A_\mathrm{hex}\mathds Z^2$}.
Reliable reference values are obtained from Corollary \ref{OneTwoBodyZetaFunction}, as the one-body zeta function vanishes whereas the two-body zeta function can be written in terms of the efficiently computable Epstein zeta function. 

The resulting error is displayed as a function of $\nu$ in the range $\nu-d\in (0,10)$ in increments of $\delta \nu=1/10$  in \Cref{fig:OneTwoBodyZetaError}. We find that for both one- and two-body zeta functions and in spatial dimensions one and two, full precision is reached against the reference values obtained from EpsteinLib. In particular, the precision of our method remains constant when increasing either $n$ or $d$.

\subsection{Three-body zeta functions, meromorphic continuation, and ATM lattice sums}

\begin{figure}
    \centering
    \includegraphics[width=1.\linewidth]{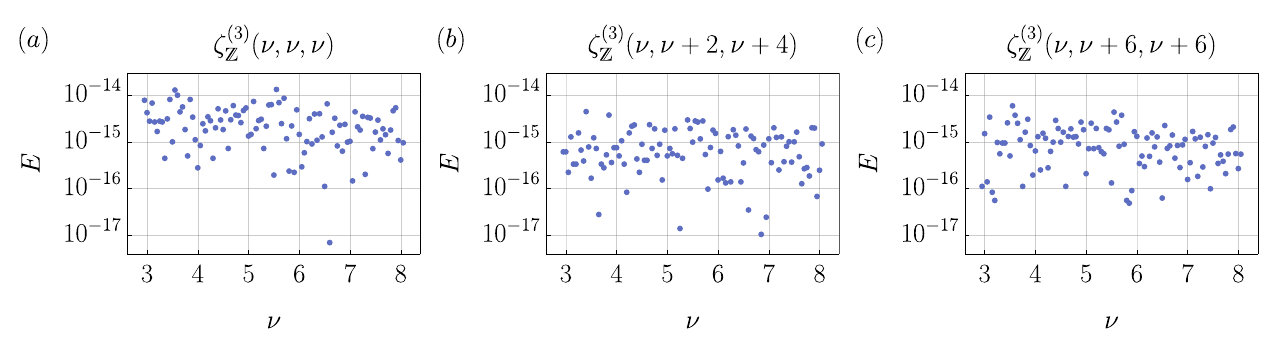}
    \caption{\new{Minimum of absolute and relative error $E$ of the three-body zeta function $\zeta^{(3)}_{\Lambda}(\bm \nu)$ on the integer lattice $\Lambda=\mathds Z$ by comparision with exact summation for sufficiently large $\bm \nu$.}}
    \label{fig:ThreeBodyErrDirect1D}
\end{figure}

For sufficiently large $\nu_i$, the three-body zeta function can be directly calculated by \Cref{multi-body-zeta} through direct summation over 
$\bm x^{(1)},\ldots,\bm x^{(n-1)}\in A\{-L,\cdots,L\}^d$ for some truncation value $L\in\mathds N$.
We compare the one-dimensional three-body zeta function  $\zeta^{(3)}_{\mathds Z}(\bm \nu)$ obtained by the integral representation with the value obtained by direct summation  for $\min \nu_i\ge 3$ and truncation value $L=1000$ in increments of $\delta \nu_i=1/20$ in \Cref{fig:ThreeBodyErrDirect1D}. Full precision is obtained over the whole parameter range. 
Direct summation for the cohesive energy of the ATM potential for the integer lattice $\Lambda=\mathds Z$ in \Cref{def:atm} for a truncation value of $L=5000$ yields $E_\mathrm{coh}^{(3)}=-0.2723018495076886$ which is in excellent agreement with the value obtained by the Epstein zeta method, yielding $E_\mathrm{coh}^{(3)}=-0.2723018495076887$.

While direct summation no longer gives an accurate reference for $\Lambda=\mathds Z$ and $\min\nu_i<3$, we may still obtain a reliable reference in one dimension using analytic high-order series expansions of the Epstein zeta function. To this end, we utilize \cite[Theorems 5.5, 5.6]{buchheit2022singular}, which allow us to write down the Taylor series of the regularized Epstein zeta function in terms of Epstein zeta functions with a shift in the argument $\nu$. The resulting expansion reads
\[
Z_{\mathds Z,\nu}( k) = \hat s_{\nu}(k)+\sum_{n=0}^\infty \frac{1}{(2k)!}\frac{k^{2n}}{(2\pi i)^{2n}} Z_{\mathds Z,\nu-2n}(0),\quad \nu\not \in 1+2\mathds N_0,
\]
where 
$Z_{\mathds Z,\nu}(0) = 2\zeta(\nu)$. After truncating the series expansion on the right-hand side at sufficiently large order, the integrals in Theorem~\ref{thm:epstein_representation} in $k$ can be evaluated analytically, providing us with the desired reference values.
The error and the function values of the one dimensional three body function
$\zeta^{(3)}_{\mathds Z}(\bm \nu)$ for steps $-2\le\nu\le 3$ in increments of $\delta\nu_i=1/20$ are shown in \Cref{fig:ThreeBodyMero}, where an offset of $1/50$ is added to the values of $\nu$, in order to avoid singularities of the many-body zeta function. Note that in contrast to the Epstein zeta function $Z_{\Lambda,\nu}(\bm k)$, which can only exhibit a simple pole in $\nu$ at $\nu=d$ for $\bm k =0$, the many-body zeta function can exhibit multiple poles in $\nu_i$ as can be seen in Fig.~\ref{fig:ThreeBodyMero} (a), (b).

Our results show that the meromorphic continuation can be reliably computed for moderately small $\nu$. In particular, the $\bm \nu$ values $(-1,5,5)^T$ and $(1,3,5)^T$ used in the computation of the ATM potential can be precisely evaluated. An unavoidable loss of precision occurs at very small values of $\nu_i$.

\begin{figure}
    \centering
    \includegraphics[width=1.\linewidth]{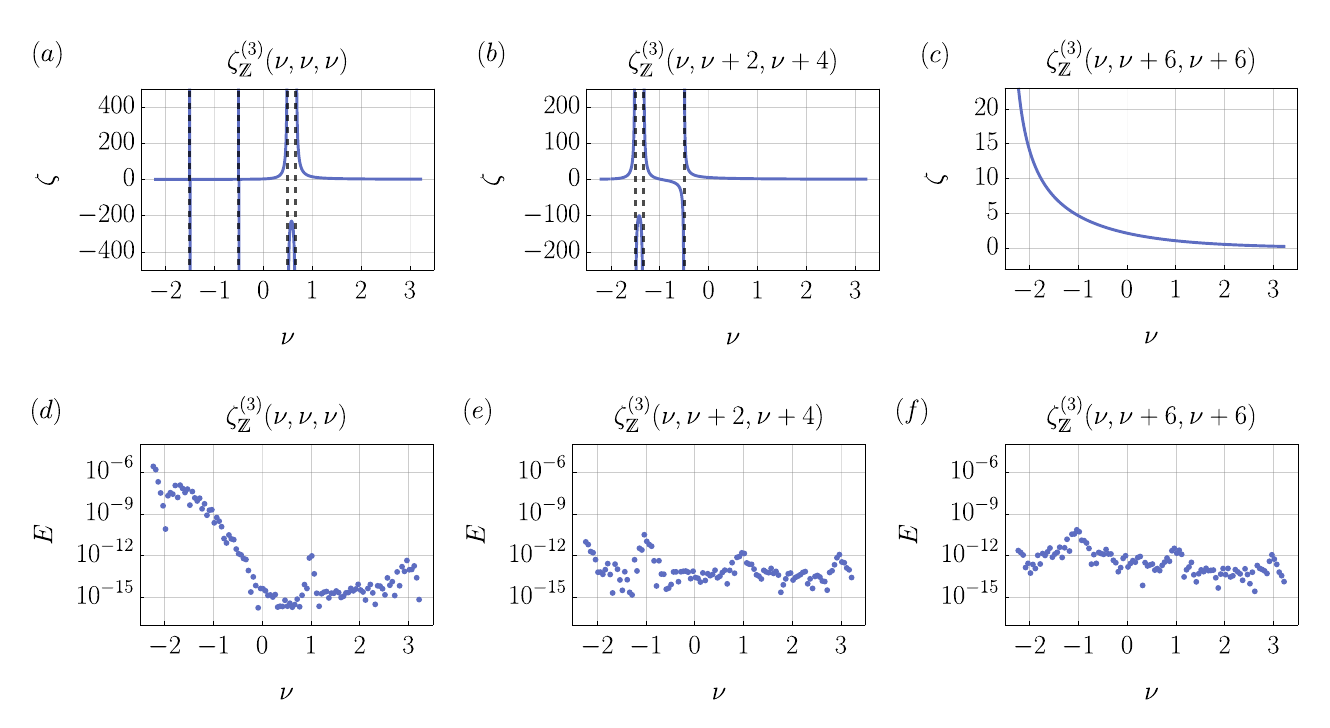}
    \caption{\new{(a)--(c) Function values of the meromorphic continuation of the three-body zeta function for $\Lambda=\mathds Z$ where the dashed black lines indicate the singularities. (d)--(f) Error values of the meromorphic continuation shown in (a)--(c). 
    (d) The error is magnified for small values of $\nu_1+\nu_2+\nu_3$.}}
\label{fig:ThreeBodyMero}
\end{figure}

As a two-dimensional benchmark of the three-body zeta function, we compare 
our two-dimensional three-body zeta function  $\zeta^{(3)}_{\Lambda}(\bm \nu)$ for $\Lambda\in\{\Lambda_{\rm sq},\Lambda_{\rm hex}\}$ to the value obtained by direct summation for $\min \nu_i\ge 5$ and truncation value $L=60$ again for increments of $\delta \nu_i=1/20$ in \Cref{fig:ThreeBodyErrDirect2D}.
Full precision is reached across the whole parameter range. In particular, all three-body zeta functions required for the ATM lattice sum can be precisely computed.
In two dimensions, direct summation no longer yields full precision for the cohesive energy within a reasonable timeframe. For the square lattice
$\Lambda = \Lambda_{\rm sq}$ with a truncation at $L = 250$, over $10^{10}$ summands are required to obtain $E_\mathrm{coh}^{(3)}=0.7700936511489661$, taking approximately 6 hours on a single core. In contrast, the Epstein zeta representation yields $E_\mathrm{coh}^{(3)}=0.7700936505168068$ within seconds, which agrees well with the high-precision result $E_{\mathrm{coh}}^{(3)}=0.77009365051710454$, obtained by direct summation for a truncation at $L=1600$ in Ref.~\cite{roblesnavarro2025exact}.

\begin{figure}
    \centering
    \includegraphics[width=1.\linewidth]{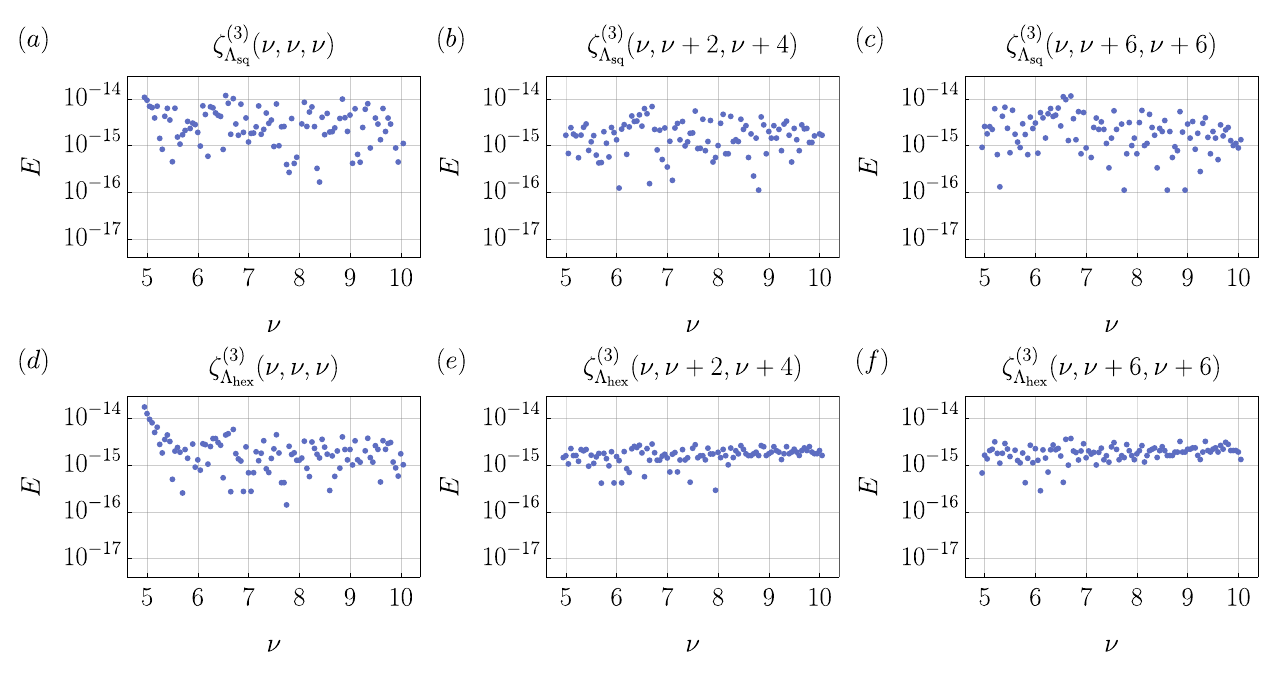}
    \caption{Minimum of absolute and relative error $E$ of the three-body zeta function $\zeta^{(3)}_{\Lambda}(\bm \nu)$  by comparison with exact summation for sufficiently large $\nu_i$ for $\Lambda=\Lambda_{\rm sq}$ in (a)--(c) and for the hexagonal lattice $\Lambda=\Lambda_{\rm hex}$ in (d)--(f).}
    \label{fig:ThreeBodyErrDirect2D}
\end{figure}

\subsection{Many-body zeta functions}

Finally, we study the many-body zeta function for large values of $n$. The arising lattice sums are of dimension $(n-1)d$ and quickly become numerically intractable by direct summation or other methods as $n$ increases. 
However, they remain highly relevant for the prediction of material properties. For the case of general $n$ and $\nu_i=3$, the resulting many-body zeta functions correspond to the isotropic part of the $n$-body cohesive energy obtained from a Drude model, see \cite[Eq.~(6)]{schwerdtfeger2016towards}.
Sums of this kind are of high interest in the perturbative treatment of quantum many-body systems, where in the past, elaborate Monte Carlo methods have been used to tackle moderately high $n$ with direct applications to quantum spin systems \cite{adelhardt2024monte}.

In the following, we test our method for $n=4$ and $n=5$ against large direct sums for $\zeta^{(n)}(\bm \nu)$ at $\bm \nu =(\nu,\ldots,\nu)^T$ for sufficiently large values of $\nu\gg d$, where the sums can be truncated at close distance to the origin. The results as a function of $\nu$ are displayed in \Cref{fig:ThreeBodyErrDirect45D}. Full precision is reached for the four-body zeta function in the parameter range $12\le\nu\le 32$  for truncation value $L=6$ and for the five-body zeta function in the parameter range $26\le \nu\le 46$ for truncation value $L=3$ in steps of size $\delta \nu=1/5$.
Here, exact summation can still be performed within acceptable timescales. The resulting sums in this range remain non-trivial, as more than nearest neighbors contribute in a non-negligible way.

\begin{figure}
    \centering
    \includegraphics[width=1.\linewidth]{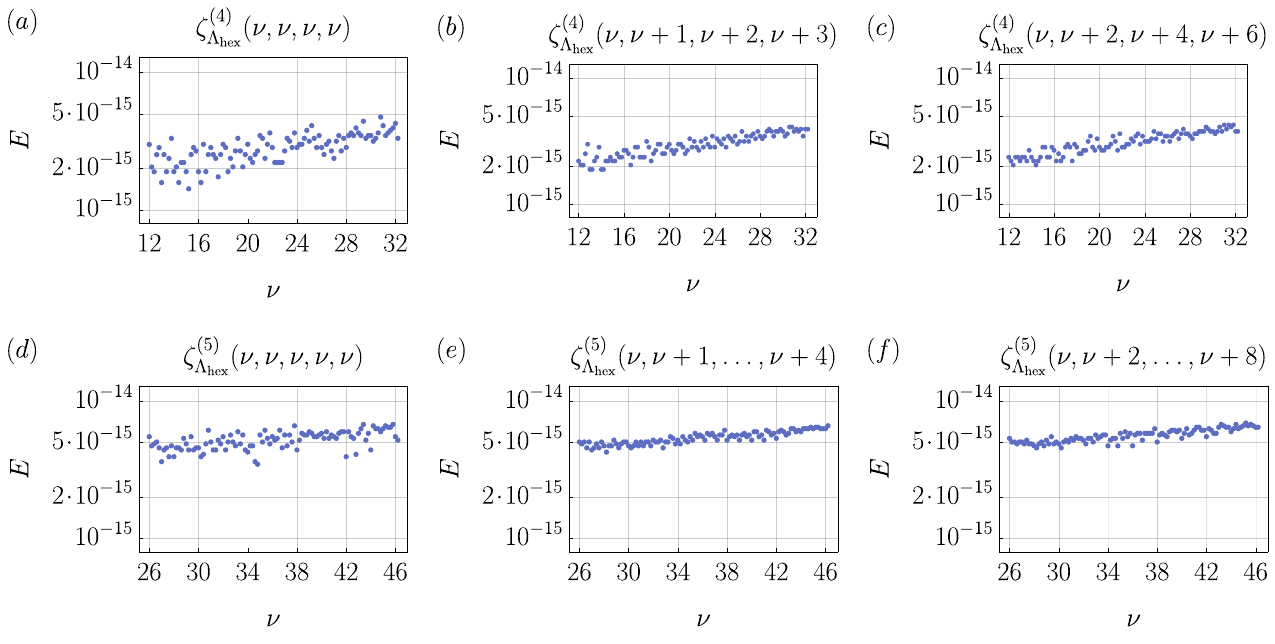}
    \caption{Minimum of absolute and relative error $E$ of the $n$-body zeta function $\zeta^{(n)}_{\Lambda}(\nu,\ldots,\nu)$ for the hexagonal lattice $\Lambda=\Lambda_{\rm hex}$ in comparison with exact summation for sufficiently large $\nu$ for $n=4$ in (a)--(c) and $n=5$ in (d)--(f).
    }
    \label{fig:ThreeBodyErrDirect45D}
\end{figure}

\begin{figure}
    \centering
    \includegraphics[width=0.47\linewidth]{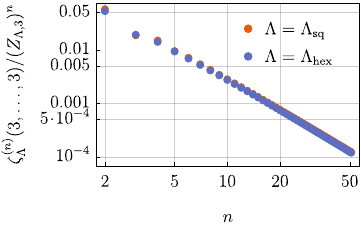}
    \caption{\new{Double-logarithmic plot of the} many-body zeta function $\zeta^{(n)}_\Lambda(\bm \nu)$ in units of the Epstein zeta function $(Z_{\Lambda,3})^n$ for $\nu_i=3$ as a function of the body number $n$ for the two-dimensional square lattice $\Lambda_\mathrm{sq}$ (black) and the hexagonal lattice $\Lambda_\mathrm{hex}$ (red). The case $n=51$ corresponds to the evaluation of a $100$-dimensional sum.}
    \label{fig:n-body-zeta}
\end{figure}

Finally, we study the behavior of the many-body zeta function as a function of $n$. As the function increases exponentially with $n$, we first establish an appropriate scaling. By inserting the bound of the Epstein zeta function, 
\[
|Z_{\Lambda,\nu}(\bm k)|=\bigg\vert \,\sideset{}{'}\sum_{\bm z\in \Lambda} \frac{e^{-2\pi i \bm z\cdot \bm k}}{\vert \bm z\vert^{\nu}}\bigg\vert\le Z_{\Lambda,\nu}(\bm 0),\quad \nu>d,
\]
into the Epstein representation in Theorem~\ref{thm:epstein_representation}, we readily obtain the bound 
\[
|\zeta^{(n)}(\bm \nu)|\le \prod_{i=1}^n Z_{\Lambda,\nu_i}(\bm 0),
\]
for $\nu_i>d$. The above bound shows that  the $n$-body zeta function increases exponentially with $n$ for exponents larger than the system dimension. A suitable normalization is thus given by the above product of Epstein zeta functions evaluated at $\bm k =\bm 0$. For ease of notation, we set $Z_{\Lambda,\nu}=Z_{\Lambda,\nu}(\bm 0)$. We show the behavior of 
\[
\zeta^{(n)}(\nu,\dots,\nu)/(Z_{\Lambda,\nu})^n
\]
as a function of $n$ in the range $n\in \{2,\dots,51\}$ for $\nu = 3$ in \Cref{fig:n-body-zeta} for the two-dimensional square lattice (black) and the hexagonal lattice (red). The case $n=2$ corresponds to a rescaled Epstein zeta function. The $n$-body zeta function then decays algebraically compared to the reference as $n$ increases. Note that the case $n=51$ corresponds to a $100$-dimensional sum, which is intractable by direct summation or  other methods. High-precision values of the $n$-body zeta function for the two-dimensional square lattice for different choices of $n$ and $\nu\in\{2,3\}$ are given in Table~\ref{tab:datasquare}.

\begin{figure}
    \centering
    \includegraphics[width=0.35\linewidth]{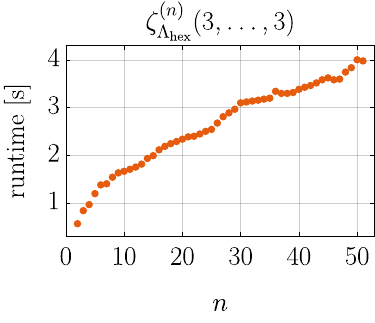}
    \caption{
    \new{Median runtime on a single core over 32 runs for a single computation of the $n$-body zeta function on the hexagonal lattice with $\nu_i=3$ as a function of $n$ on an Apple M1 Max processor. The case $n=51$ corresponds to a $100$-dimensional sum. We observe a  linear increase of the runtime in $n$, whereas the numerical work for direct summation approaches increases exponentially in the number of interacting bodies.}}
    \label{fig:runtime}
\end{figure}

\pgfplotstableread{tabSquareSmall.dat}\datatable

\begin{table}[ht]
\centering
\begin{minipage}{0.7\textwidth} 
    \centering 
    \pgfplotstabletypeset[
        columns={0,1,2}, 
        columns/0/.style={column name={$n$}},
        columns/1/.style={column name={$\nu=2$}},
        columns/2/.style={column name={$\nu = 3$}},
        every head row/.style={before row=\hline, after row=\hline},
        every last row/.style={after row=\hline},
        precision=15
    ]{\datatable}
    \caption{Values of the many-body zeta function $\zeta^{(n)}_{\Lambda}(\nu,\dots,\nu)$ for the two-dimensional square lattice for different choices of $n$ and $\nu$.}
    \label{tab:datasquare}
\end{minipage}
\end{table}

We finally study the runtime of our algorithm on 8 cores on an Apple M1 Max CPU as a function of $n$ for the two-dimensional hexagonal lattice and $\bm \nu = (3,\dots 3)^T$ in \Cref{fig:runtime}. We find that the runtime increases at most linearly with $n$, in contrast to the exponential scaling of direct summation. This is due to the singularity of the Epstein zeta function scaling as $\vert\bm k \vert^{\nu-d}$
for $\nu-d>0$. Thus products of the singularities exhibit more regularity than their constitutes and the term that determines the complexity of the integration task is thus the single singularity with the smallest value of $\nu_i$. 
For this reason, the required number of quadrature points after Duffy transformation and thus the computational cost remains effectively constant and the computational effort only increases linearly with $n$ due to the larger number of Epstein zeta functions to be evaluated. The computation of a $100$-dimensional sum on a two-dimensional hexagonal lattice requires less than $4$ seconds with the Epstein representation. Our method thus allows to compute sums at extremely high dimensions. This opens up new possibilities for both high-order and high-precision perturbative treatments of quantum many-body systems. 

\section{Influence of three-body interactions on the stability of 3D crystal lattices}
\label{sec:application}

\begin{figure}
    \centering
    \includegraphics[width=.7\linewidth]{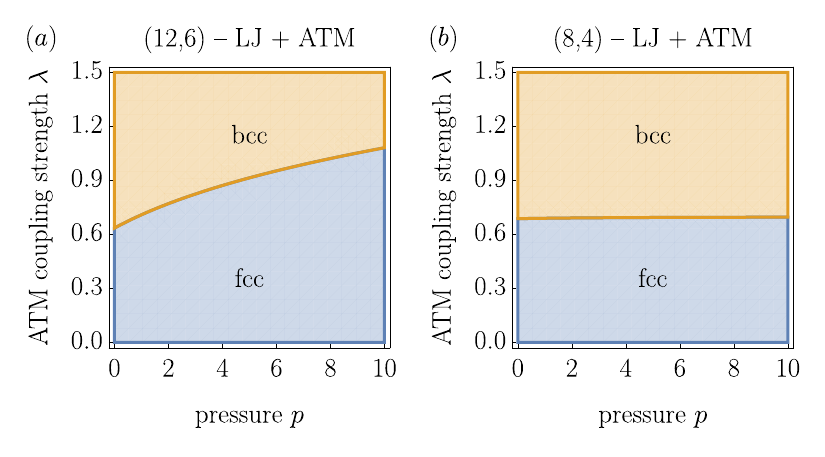}
    \caption{Phase transition at $T=0$ in three-dimensions between the face-centered cubic (fcc) lattice  (blue) and body-centered cubic (bcc) lattice (orange) for a $(12,6)$ Lennard Jones potential in (a) and an $(8,4)$ Lennard-Jones potential in (b) as a function of the pressure $p$ and the ATM coupling strength $\lambda$.}
    \label{fig:pressure}
\end{figure}

We will now apply our new method to study the influence of three-body interactions on the stability of matter. To this end, we study a three-dimensional Bravais lattice $\Lambda$ of atoms that interact via a two-body Lennard-Jones potential  of the form \cite{schwerdtfeger2024onehundred}
\[
U^{(2)}_{\mathrm{LJ}}(\bm r) = \varepsilon\frac{nm}{(n-m)}\bigg(\frac{r_e^n}{n \vert \bm r\vert^n}-\frac{r_e^m}{m \vert \bm r\vert^{m}} \bigg),
\]
with the dissociation energy $\varepsilon>0$ and the equilibrium distance $r_e>0$. In the following, we adopt dimensionless units, writing distances in units of $r_e$ and energies in terms of $\varepsilon$. In addition, we consider a three-body ATM potential as in Def.~\ref{def:atm} with coupling strength $\lambda>0$. The total cohesive energy per particle thus reads
\[
E_\mathrm{coh} = \frac{1}{2}\,\sideset{}{'}\sum_{\bm x\in \Lambda}U^{(2)}_{\mathrm{LJ}}(\bm x)+\frac{\lambda}{6}\,\sideset{}{'}\sum_{\bm x,\bm y\in \Lambda}U^{(3)}_\mathrm{ATM}(\bm x,\bm y) =E_\mathrm{coh}^{(2)}+ E_\mathrm{coh}^{(3)}.
\]
Finally, for finite pressure $p$, the enthalpy per particle is given by
\begin{equation*}
H=E_\mathrm{coh}+pV_\Lambda.
\end{equation*}
We now consider the normalized bcc and fcc lattices\new{, which we recall are} 
\[
\Lambda_\mathrm{fcc}=A_\mathrm{fcc} \mathds Z^3,\quad \Lambda_\mathrm{bcc}=A_\mathrm{bcc}\mathds Z^3
\] with unit nearest-neighbor distance, defined by their respective lattice matrices 
\[
A_\mathrm{fcc}=\frac{1}{\sqrt{2}} \begin{pmatrix}
    1 & 1 & 0\\
    1 & 0 & 1\\
    0 & 1 & 1
\end{pmatrix},\quad A_\mathrm{bcc}=\frac{1}{\sqrt{3}} \begin{pmatrix}
    1 & 1 & 0\\
    \sqrt{2} & 0 & \sqrt{2}\\
    0 & \sqrt{2} & \sqrt{2}
\end{pmatrix}.
\]
The corresponding lattices with nearest-neighbor distance $R>0$ are then obtained through rescaling. The two-body interaction energy at finite distance $R$ then follows from the Epstein zeta function as 
\[
E_\mathrm{coh}^{(2)} = \frac{nm}{2(n-m)}\bigg(\frac{Z_{\Lambda,n}(\bm 0)}{n R^n }-\frac{Z_{\Lambda,\new{m}}(\bm 0)}{m R^m} \bigg),
\]
and the three-body cohesive energy in terms of three-body zeta functions reads
\[
E_\mathrm{coh}^{(3)}=\frac{\lambda}{R^9}\bigg(\frac{1}{24}\zeta_\Lambda^{(3)}(3,3,3) - \frac{3}{16} \zeta_\Lambda^{(3)}(-1,5,5)+\frac{3}{8} \zeta_\Lambda^{(3)}(1,3,5)\bigg).
\]
We minimize the enthalpy with respect to $R$ for the fcc and the bcc lattice and compare the resulting enthalpies as a function of pressure $p$ and ATM coupling constant $\lambda$. The resulting phase transition is displayed in \Cref{fig:pressure} (a) for the $(12,6)$-LJ potential and in (b) for the softer $(8,4)$-LJ potential. We observe that for vanishing $\lambda$, the fcc phase is always energetically favorable. As $\lambda$ increases, a critical value is reached where bcc exhibits a lower enthalpy than fcc. Thus, for sufficiently large coupling strengths $\lambda$, three-body interactions can destabilize fcc and lead to the formation of a bcc phase. As a function of pressure, we observe that larger pressures favor the fcc phase, due to the smaller value of $V_\Lambda$ and thus larger packing density. While a strong dependency on $p$ of the phase boundary is observed for the $(12,6)$ potential, the phase boundary in the case of an $(8,4)$ potential only exhibits a minute dependency on the pressure $p$. Our method forms the numerical and analytical foundation of Ref.~\cite{roblesnavarro2025exact}, where the influence of three-body interactions on the stability of crystal lattices is investigated in detail.

\section{Outlook}
\label{sec:outlook}

This work solves the challenge of efficiently computing a general class of many-body lattice sums, including their meromorphic continuations, in terms of many-body zeta functions. Our method is applicable to the highly relevant three-body Axilrod-Teller-Muto potential and to a general class of multi-body interactions.  With our approach, sums for three-dimensional lattices, which traditionally require weeks of computational time on a single CPU, using standard techniques, can now be evaluated to machine precision within minutes on a laptop. Our method is general and is applicable to any lattice, any power-law interaction, and any system dimension. It allows to precisely compute the minute energy differences between lattice structures due to three-body interactions, forming the foundation for their rigorous treatment in computational chemistry in \cite{roblesnavarro2025exact}. Further, the runtime for computing $n$-body lattice sums scales at most linearly with $n$, in contrast to exponential scaling for direct summation. For a two-dimensional lattice and $n=51$, our method allows for the computation of physically relevant $100$-dimensional sums within seconds on a standard laptop. 

Accurately incorporating many-body interactions is crucial for advancing our understanding of condensed matter systems. In theoretical chemistry, our method will help explain why certain crystal structures  are stable in nature while others are not. In condensed matter physics, our method can readily be used in the perturbative study of quantum spin systems with long-range interactions, see e.g.~\cite{adelhardt2024monte, fey2019quantum}. \new{We also note the applicability of our method to high-temperature expansion for quantum spin systems \cite{burkard2026dynamic,burkard2026high}.}
Looking ahead, we will extend our method to even more general interaction graphs with direct applications to high-order perturbative treatments of quantum many-body systems.

\section*{Acknowledgements}

A.B. expresses his sincere gratitude to Peter Schwerdtfeger for inspiring discussions and for introducing him to the underlying problem in theoretical chemistry that motivated this work. He also thanks Torsten Keßler, Kirill Serkh, and Daniel Seibel for their insightful input on singular quadrature, and Jan Koziol for valuable discussions on quantum spin systems. A.B. and J.B. are grateful to Ruben Gutendorf for his important contributions to EpsteinLib. A.B. and J.B. would like to thank Sergej Rjasanow, whose unwavering support made this work possible.

\section*{Declarations}

\subsection*{Funding}
\new{This work was supported by the Klaus-Tschira Stiftung under Grant No. 00.025.2025. }
The authors gratefully acknowledge the scientific support and HPC resources provided by the Erlangen National High Performance Computing Center (NHR@FAU) of the Friedrich-Alexander-Universit\"at Erlangen-N\"urnberg (FAU) under the NHR project n101af. NHR funding is provided by federal and Bavarian state authorities. NHR@FAU hardware is partially funded by the German Research Foundation (DFG)-440719683. 
J.B. acknowledges the support of the Quantum Fellowship Program of the German Aerospace Center (DLR) for funding their contribution to this work.

\subsection*{Competing interests}
The authors declare that they have no competing interests.

\subsection*{Code availability}
The open source library EpsteinLib is freely available on \href{https://github.com/epsteinlib/epsteinlib}{GitHub}. 

\subsection*{Author contribution}
All authors contributed equally to this work.

\printbibliography

\end{document}